\documentclass[12pt, twoside, a4paper]{amsart}
\usepackage{amsmath, amsthm, amsopn, amssymb, a4wide, varioref, amscd}

\usepackage{color}
\usepackage{bm}
\usepackage{hyperref}
\usepackage{tikz-cd}
\newcommand{\comment}[1]{}
\newcommand{\bydef}{\stackrel{\rm def}{=}}

\newcommand{\cle}{\mathcal{E}}
\newcommand{\clh}{\mathcal{H}}

\theoremstyle{theorem}
    \newtheorem{theorem}{Theorem}
    \newtheorem{lemma}{Lemma}
    \newtheorem{proposition}[lemma]{Proposition}
    \newtheorem{corollary}[lemma]{Corollary}

    \newtheorem{definition}[lemma]{Definition}
    
\theoremstyle{definition}

    \newtheorem{remark}[lemma]{Remark}
    \newtheorem{example}[lemma]{Example}
    \newtheorem{exercise}[lemma]{Exercise}

\newcommand\mnote[1]{} 
\newcommand\be{\begin{equation*}}

\newcommand\ee{\end{equation*}}

\newcommand\ben{\begin{equation}}
\newcommand\een{\end{equation}}
\newcommand\bes{\begin{eqnarray*}}
\newcommand\ees{\end{eqnarray*}}

\newcommand\bex{\begin{exercise}}
\newcommand\eex{\end{exercise}}
\newcommand\beg{\begin{example}}
\newcommand\eeg{\end{example}}
\newcommand\benu{\begin{enumerate}}
\newcommand\eenu{\end{enumerate}}
\newcommand\beit{\begin{itemize}}
\newcommand\eeit{\end{itemize}}
\newcommand\berk{\begin{remark}}
\newcommand\eerk{\end{remark}}
\newcommand\bdefn{\begin{defintion}}
\newcommand\edefn{\end{definition}}
\newcommand\bthm{\begin{theorem}}
\newcommand\ethm{\end{theorem}}
\newcommand\blem{\begin{lemma}}
\newcommand\elem{\end{lemma}}

\newcommand{\sm}{{\raise0.3ex\hbox{$\scriptstyle \setminus$}}}

\def\sbar{{\overline{s}}}
\def\pbar{{\overline{p}}}

\def\CHI{\mathchoice%
{\raise2pt\hbox{$\chi$}}%
{\raise2pt\hbox{$\chi$}}%
{\raise1.3pt\hbox{$\scriptstyle\chi$}}%
{\raise0.8pt\hbox{$\scriptscriptstyle\chi$}}}
\def\smalloplus{\raise1pt\hbox{$\,\scriptstyle \oplus\;$}}

\numberwithin{equation}{section}

\newcommand{\cl}[1]{\mathcal{#1}}

\begin{document}

\title[Toeplitz operators on the symmetrized bidisc]{Toeplitz operators on the symmetrized bidisc}

\author[Bhattacharyya]{Tirthankar Bhattacharyya}
\address{Department of Mathematics, Indian Institute of Science, Bangalore 560012, India}
\email{tirtha@iisc.ac.in}

\author[Das] {B. Krishna Das}
\address{Department of Mathematics, Indian Institute of Technology Bombay, Powai, Mumbai, 400076, India}
\email{dasb@math.iitb.ac.in, bata436@gmail.com}

\author[Sau]{Haripada Sau}
\address{Department of Mathematics, Indian Institute of Technology Guwahati, Guwahati, Assam 781039, India}
\email{haripadasau215@gmail.com}


\begin{abstract}

  The symmetrized bidisc has been a rich field of holomorphic function theory and operator theory. A certain well-known reproducing kernel Hilbert space of holomorphic functions on the symmetrized bidisc resembles the Hardy space of the unit disc in several aspects. This space is known as the Hardy space of the symmetrized bidisc. We introduce the study of those operators on the Hardy space of the symmetrized bidisc that are analogous to Toeplitz operators on the Hardy space of the unit disc. More explicitly, we first study multiplication operators on a bigger space (an $L^2$-space) and then study compressions of these multiplication operators to the Hardy space of the symmetrized bidisc and prove the following major results.

  \begin{enumerate}

  \item   Theorem I  analyzes the Hardy space of the symmetrized bidisc, not just as a Hilbert space, but as a Hilbert module over the polynomial ring and finds three isomorphic copies of it as $\mathbb D^2$-contractive Hilbert modules.

  \item Theorem II provides an algebraic, Brown and Halmos type,  characterization of Toeplitz operators.

  \item Theorem III gives several characterizations of an analytic Toeplitz operator.

  \item Theorem IV characterizes asymptotic Toeplitz operators.

  \item Theorem V is a commutant lifting theorem.

  \item Theorem VI yields an algebraic characterization of dual Toeplitz operators.

\end{enumerate}

Every section from Section 1 to Section 6 contains a theorem each, the main result of that section.

\end{abstract}
\maketitle

\section{$\Gamma$ and $\Gamma$-contractions - preliminaries}

Ever since Brown and Halmos published their seminal paper (\cite{BH}) on Toeplitz operators, it has been vastly studied. The book by Bottcher and Silverman (\cite{BS}) is a veritable treasure. For the introduction to the theory for just the space $H^2(\mathbb D)$, the survey article by Axler (\cite{A}) is excellent. State of the art research, even just in the context of the unit disc $\mathbb D = \{ z \in \mathbb C : |z| < 1 \}$ is still going on, see \cite{CHKL}, \cite{Englis} and \cite{Yakub} and there are open problems, see \cite{Lee-open}. Toeplitz operators have found applications in a wide variety of areas of mathematics from algebraic geometry (\cite{Rietsch}) to operator algebras (\cite{DH}).

In several variables, Toeplitz operators have been studied by several authors, see \cite{MSS} and the references therein. Naive attempts to generalize one variable results quickly run into difficulties and innovative
new ideas are required.

The open symmetrized bidisc is defined as
 $$ \mathbb G = \{ (z_1 + z_2, z_1z_2) : |z_1| < 1 \mbox{ and } |z_2| < 1\}.$$
 The novelty of this domain arises from the fact that it behaves significantly differently from even the bidisc (e.g., a realization formula for a function in the unit
 ball of $H^\infty(\mathbb G)$ requires uncountably infinitely many ``test functions", see \cite{AY-R} and \cite{bhat-me realization} or see \cite{ALY} for a description of the sets with the extension property).
 The Toeplitz operators on this domain will highlight a few similarities and a lot of differences with the classical situation of Brown and Halmos as well as with later endeavours on the bidisc.
 It will also bring out once again the importance of the fundamental operator of a $\Gamma$-contraction introduced in \cite{BhPSR}. Let $\Gamma$ denote the closed symmetrized bidisc $ \Gamma = \{ (z_1 + z_2, z_1z_2) : |z_1| \le 1 \mbox{ and } |z_2| \le 1\}.$ The following terminology is due to Agler and Young, \cite{AY-JOT}.

\begin{definition}
Let $b\Gamma$ be the distinguished boundary of the symmetrized bidisc,
i.e., $b\Gamma = \{ (z_1 + z_2, z_1 z_2) : |z_1| = |z_2| = 1\}$.
 \begin{enumerate}

\item A commuting pair $(R,U)$ is called a $\Gamma$-{\em unitary}
if $R$ and $U$ are normal operators and the joint spectrum
$\sigma(R,U)$ of $(R,U)$ is contained in the distinguished
boundary of $\Gamma$.

\item
A commuting pair $(T,V)$ acting on a Hilbert space $\mathcal K$ is called a $\Gamma$-isometry if there
exist a Hilbert space $\mathcal{N}$ containing $\mathcal{K}$ and a
$\Gamma$-unitary $(R,U)$ on $\mathcal{N}$ such
that $\mathcal{K}$ is left invariant by both $R$ and
$U$, and
$$T = R|_{\mathcal{K}} \mbox{ and } V = U|_{\mathcal{K}}.$$

\end{enumerate}

\end{definition}

In other words, $(R,U)$ is a $\Gamma$-unitary
extension of $(T,V)$. In block operator matrix form,
$$ R = \left(
                 \begin{array}{cc}
                   T & * \\
                   0 & * \\
                 \end{array}
               \right) \mbox{ and } U = \left(
                 \begin{array}{cc}
                   V & * \\
                   0 & * \\
                 \end{array}
               \right) $$
with respect to the decomposition $\mathcal{N}=\mathcal{K}\oplus\mathcal{K}^{\perp}$.

A $\Gamma$-isometry $(T,V)$ on $\mathcal H$ is said to be a pure $\Gamma$-isometry
if $V$ is a pure isometry, i.e., there is no non trivial subspace
of $\mathcal H$ on which $V$ acts as a unitary operator.

It is clear from the block matrices above that for any polynomial $\xi$ in two variables,
$$
\xi(R , U) = \left(
                 \begin{array}{cc}
                   \xi(T,V) & * \\
                   0 & * \\
                 \end{array}
               \right) .$$
               Consequently, if $\| f \|_{\infty, \Gamma}$ denotes the supremum norm of $f$ over the compact set $\Gamma$ for a function holomorphic in a neighbourhood of $\Gamma$, then for any polynomial $\xi$,
                \begin{eqnarray}
                \| \xi(T,V)\| & \le & \| \xi(R, U) \|  \nonumber \\
                & = & r(\xi(R, U)) \mbox{ (because of normality)} \nonumber \\
                & = & \sup\{ |\xi(s,p)| : (s,p) \in \sigma(R, U) \} \nonumber \\
                & \le & \sup\{ |\xi(s,p)| : (s,p) \in b\Gamma \} \mbox{ (because } \sigma(R, U) \subseteq b\Gamma \text{)} \nonumber \\
                & = & \| \xi \|_{\infty, \Gamma}. \nonumber
                \end{eqnarray}

                This von Neumann type inequality will also remain true for another class of operator pairs $(S,P)$. Suppose $\mathcal H$ is a subspace of $\mathcal K$ that is invariant under $T^*$ and $V^*$. On $\mathcal H$, we consider the operators $S$ and $P$ which are defined by
                \begin{equation} \label{co-invariant} S^* = T^*|_{\mathcal H} \mbox{ and } P^* = V^* |_{\mathcal H}. \end{equation}
                So, $S$ and $P$ are compressions of $T$ and $V$ to a co-invariant subspace. In block operator matrix form with respect to the orthogonal decomposition $\mathcal K = \mathcal H \oplus (\mathcal K \ominus \mathcal H)$, we have
$$ T = \left(
                 \begin{array}{cc}
                   S & 0 \\
                   * & * \\
                 \end{array}
               \right) \mbox{ and } V = \left(
                 \begin{array}{cc}
                   P & 0 \\
                   * & * \\
                 \end{array}
               \right) . $$
               If $\xi(s,p) = \sum a_{ij} s^i p^j$ is a polynomial, then because of the structure of the block matrices above,
               $$ \xi(T,V) = \left(
                 \begin{array}{cc}
                   \xi(S,P) & 0 \\
                   * & * \\
                 \end{array}
               \right).$$
Thus,
                 \begin{equation} \label{GammaC} \| \xi(S,P) \| = \| P_{\mathcal H}\xi(T, V) \| \le \| \xi(T,V) \| \le \| \xi \|_{\infty, \Gamma}. \end{equation}
Since $\Gamma$ is polynomially convex and since the inequality \eqref{GammaC} holds for all polynomials, the Oka-Weil Theorem implies that the same holds for all $f \in A(\Gamma)$. Thus starting with a co-invariant subspace $\mathcal H$ of a $\Gamma$-isometry $(T,V)$, we showed that the compression pair $(S,P)=P_{\mathcal H}(T,V)|_{\mathcal H}$ satisfies the inequality \eqref{GammaC}. {\em It is a remarkable fact that the converse is true}, i.e., given any commuting pair $(S,P)$ of bounded operators on a Hilbert space $\mathcal H$ satisfying the inequality
                  $$ \| \xi(S,P) \| \le \| \xi \|_{\infty, \Gamma}$$
for all polynomials $\xi$ in two variables (equivalently, for all $f \in A(\Gamma)$ because of the Oka-Weil Theorem), there is a bigger Hilbert space $\mathcal K$ containing $\mathcal H$ and a $\Gamma$-isometry $(T,V)$ acting on $\mathcal K$ such that $\mathcal H$ is a joint co-invariant subspace for $(T, V)$ ($T^*\mathcal H \subset \mathcal H$ and $V^*\mathcal H \subset \mathcal H$) and $(S,P)$ and $(T,V)$ satisfy (\ref{co-invariant}). This is the Agler-Young dilation of a $\Gamma$-contraction, discovered and expounded upon in \cite{AY}, \cite{ay-gamma} and \cite{AY-JOT}.

                  \begin{definition}
                    A pair of commuting bounded operators $(S,P)$ on a Hilbert space $\clh$ is called a $\Gamma$-contraction if $$ \| \xi(S,P) \| \le \| \xi \|_{\infty, \Gamma}$$
                  for all polynomials $\xi$ in two variables. \label{GammaCOntractionDefinition} \end{definition}
 We saw in the paragraph preceding the definition that every $\Gamma$--contraction $dilates$, first to a $\Gamma$--isometry and then to a $\Gamma$--unitary. Thus the structures of these two classes of operator pairs become important. The two following propositions are collections of results from \cite{AY-JOT} and \cite{BhPSR} and characterize $\Gamma$-unitaries and $\Gamma$-isometries.

\begin{proposition} \label{G-unitary}
Let $\mathcal{H}$ be a Hilbert space and let $R, U \in \mathcal B (\mathcal H)$ satisfy $RU = UR$. Then the following are equivalent:
\begin{enumerate}

\item $(R,U)$ is a $\Gamma$-unitary;
\item there exist commuting unitary operators $U_{1}$ and $U_{2}$ on $\mathcal{H}$ such that
$$R= U_{1}+U_{2},\quad U= U_{1}U_{2};$$
\item $U$ is unitary,\;$R=R^*U,\;$\;and $r(R)\leq2,$ \; where
    $r(R)$ is the spectral radius of $R$;
    \item $(R,U)$ is a $\Gamma$-contraction and $U$ is a
        unitary;
    \item $U$ is a unitary and $R = W + W^* U$ for some
        unitary $W$ commuting with $U$.
\end{enumerate}
\end{proposition}

\begin{proposition} \label{G-isometry}
Let $\mathcal{H}$ be a Hilbert space and let $T, V \in \mathcal B (\mathcal H)$ satisfy $TV = VT$. The following statements are equivalent:
\begin{enumerate}
\item $(T,V)$ is a $\Gamma$-isometry;
\item $(T,V)$ is a $\Gamma$-contraction and $V$ is isometry;
\item $V$ is an isometry\;,\;$T=T^*V$ and $r(T)\leq2.$

\end{enumerate}
\end{proposition}

                   Note that if $(S,P)$ is a $\Gamma$-contraction, then $P$ is a contraction. For a contraction $P$, the space $\mathcal D_P$ denotes the closure of the range of the defect operator $D_P:=(I-P^*P)^{1/2}$ of $P$.

The discovery of the {\em fundamental operator} $F$ of a $\Gamma$-contraction $(S,P)$ in \cite{BhPSR} changed the subject because with the help of it, one produces the $\Gamma$--isometric dilation, alluded to above, explicitly; characterizes $\Gamma$-contractions (Theorem 4.4 in \cite{BhPSR}); constructs a functional model (Theorem 4.4 in \cite{BhP}) and characterizes distinguished varieties in the symmetrized bidisc, see \cite{PS}. The fundamental operator is the unique bounded operator on $\mathcal D_P$ that satisfies the equation
$$S - S^*P = D_P F D_P.$$
Since its discovery, it has proved to be an indispensable tool in the study of operator theory on the symmetrized bidisc. The fundamental operator appears in this paper in Example \ref{FO} and also in
Proposition \ref{charc-compact} while characterizing compact operators on $H^2(\mathbb G)$.

 \section{The Hardy space, boundary values and Toeplitz operators}
 The beginning of this section warrants a discussion on Hilbert modules over polynomial rings. A {\em Hilbert module} over the polynomial ring $\mathbb C[z_1, z_2]$ is a Hilbert space $\clh$ which is also a module over $\mathbb C[z_1, z_2]$. If $\Omega$ is a domain in $\mathbb C^2$, then a Hilbert module $\clh$ is said to be {\em $\Omega$-contractive} if $ \| \xi \cdot h \| \le \| \xi \|_{\infty, \Omega} \|h\|$ for all $\xi$ in $\mathbb C[z_1, z_2]$ and $h$ in $\clh$. For example, by virtue of Ando's theorem, a pair of commuting contractions $T_1$ and $T_2$ acting on a Hilbert space $\clh$ makes $\clh$ a $\mathbb D^2$-contractive Hilbert module if we define
 \begin{equation} \label{D2contractivemodule} \xi \cdot h = \xi(T_1, T_2)h, \text{for all polynomials } \xi \text{ in two variables and } h \in \clh. \end{equation}
 Conversely, any $\mathbb D^2$-contractive Hilbert module gives rise to a pair of commuting contractions $T_1$ and $T_2$ such that the module action agrees with \eqref{D2contractivemodule} above. Indeed, just define $T_i h  = z_i \cdot h$ for $h$ in $\clh$ and $i=1,2$.  We shall be concerned with four Hilbert modules over the polynomial ring in two variables. The contractivity conditions will be over the bidisc $\mathbb D^2$. These specific $\mathbb D^2$-contractive Hilbert modules that we are concerned with will appear towards the end of this section because the appropriate spaces and the commuting pairs of contractions need to be introduced first.

Let $\pi$ be the symmetrization map
\begin{eqnarray}\label{pi}\pi (z_1, z_2) = (z_1+z_2 , z_1z_2),\end{eqnarray}
and $J$ be the complex Jacobian of $\pi$, i.e., $J(z_1,z_2)=z_1-z_2$ and $\mathbb T = \{\alpha:|\alpha| = 1\}$.
\begin{definition}\label{GHardy}

The Hardy space $H^2(\mathbb G)$ of the symmetrized bidisc is the vector space of those holomorphic functions $f$ on $\mathbb G$ which satisfy
\[\sup_{\,0<r<1}\int_{\mathbb T \times \mathbb T} |f\circ \pi (r\zeta_1, r\zeta_2)|^2
|J(r\zeta_1, r\zeta_2)|^2 dm(\zeta_1,\zeta_2) <\infty \]
where $m$ is the measure on the torus $\mathbb T \times \mathbb T$ obtained by taking product of the normalized arc length measure on the unit circle $\mathbb T$ with itself.
The norm of $f\in H^2(\mathbb G)$ is defined to be
$$\|f\|= \|J\|^{-1}\Big \{\sup_{0<r<1}\int_{\mathbb
T \times \mathbb T}|f\circ \pi (r\zeta_1, r\zeta_2)|^2 |J(r\zeta_1, r\zeta_2)|^2
dm(\zeta_1,\zeta_2) \Big \}^{1/2},$$
where $\|J\|^2={{\int_{\mathbb T \times \mathbb T}}} |J(\zeta_1,\zeta_2)|^2 dm(\zeta_1,\zeta_2) = 2.$
\end{definition}
In the expression of $\|f\|$, we divide by $\|J\|$ to ensure that the norm of the function $1$ in $H^2(\mathbb G)$ is $1$. This space has been discussed before for other purposes in \cite{bhat-me realization}.
Our first result establishes boundary values of the Hardy space functions. To that end, first consider the measure $\mu$ on the $2$-torus $\mathbb T \times \mathbb T$ defined, for a Borel subset $F$ of $\mathbb T \times \mathbb T$, as
$$
\mu(F):=\int_{F}|J(\zeta_1,\zeta_2)|^2dm(\zeta_1,\zeta_2).
$$
We then consider the push forward measure on $b\Gamma$ via the symmetrization map $\pi$:
$$\nu(E) = \mu(\pi^{-1}(E)) \text{ for every Borel subset } E \text{ of } b\Gamma.$$
We are now ready to define the {\em $L^2$-space over $b\Gamma$} with respect to this push-forward measure:
\begin{align*}
L^2(b\Gamma) & = \{f:b\Gamma \to \mathbb C: \int_{b\Gamma} |f|^2d\nu <\infty \}\\
& = \{f:b\Gamma \to \mathbb C: \int_{\mathbb T \times \mathbb T} |f(\pi(\zeta_1, \zeta_2))|^2 |J(\zeta_1, \zeta_2)|^2 dm(\zeta_1, \zeta_2)  <\infty \}.\end{align*}

The following embedding lemma immediately allows us to consider boundary values of the Hardy space functions.
\begin{lemma} \label{Iso1}
There is an isometric embedding of the space $H^2(\mathbb G)$ inside $L^2(b\Gamma)$.
\end{lemma}

\begin{proof}
Consider the subspace
$$H^2_{\rm anti} (\mathbb D^2) \bydef \{ f \in H^2(\mathbb D^2) : f(z_1 , z_2) = - f(z_2 , z_1) \}$$
of anti-symmetric functions of the Hardy space over the bidisc
$$H^2(\mathbb D^2)=\{ f: \mathbb{D}^2 \to \mathbb{C}: f(z_1,z_2)=\sum_{i=0}^\infty\sum_{j=0}^\infty a_{i,j}z_1^iz_2^j \text{ with } \sum_{i=0}^\infty\sum_{j=0}^\infty |a_{i,j}|^2 < \infty \}.
$$
Suppose $L^2_{\rm anti} (\mathbb T^2)$ is the subspace of $L^2(\mathbb T^2)$ consisting of anti-symmetric functions, i.e.,
$$f(\zeta_1,\zeta_2)=-f(\zeta_2,\zeta_1)\text{ a.e.}.$$
Define $\tilde{U}: H^2(\mathbb G) \to H^2_{\rm anti} (\mathbb D^2)$ by
\begin{eqnarray}\label{utilde}
\tilde{U}(f) = \frac{1}{\| J \|} J(f \circ \pi), \text{ for all } f \in H^2(\mathbb G)
\end{eqnarray}
and ${U}:L^2(b\Gamma) \to L^2_{\rm anti} (\mathbb D^2)$ by
\begin{eqnarray}\label{unotilde}
 U f=\frac{1}{\|J\|} J (f \circ \pi), \text{ for all } f \in L^2(b\Gamma).
\end{eqnarray}
It is easy to see that $U$ and $\tilde U$ are indeed unitary operators. Also note that there is an isometry $W: H^2_{\rm anti} (\mathbb D^2) \to L^2_{\rm anti}(\mathbb T^2)$ which sends a function to its radial limit. Therefore we have the following commutative diagram:
$$
\begin{CD}
H^2(\mathbb G) @> U^{-1}\circ W \circ \tilde{U}>>  L^2(b\Gamma)\\
@V\tilde{U} VV @VV U V\\
H^2_{\rm anti} (\mathbb D^2)@>>W> L^2_{\rm anti}(\mathbb T^2)
\end{CD}.
$$
Hence the map that places $H^2(\mathbb G)$ isometrically into $L^2(b\Gamma)$ is $ U^{-1}\circ W \circ \tilde{U}$.
%
\end{proof}
The above identification theorem reveals that the isometric image of the Hardy space of the symmetrized bidisc is precisely the following space:
 $$
 \{f\in L^2(b\Gamma):  U(f) \text{ has all the negative Fourier coefficients zero}\}.
 $$In this paper, we shall not make any distinction between these two realizations of the Hardy space of the symmetrized bidisc and $Pr$ will stand for the orthogonal projection of $L^2(b \Gamma)$ onto the isometric image of $H^2(\mathbb G)$ inside $L^2(b \Gamma)$. With this identification, the unitary $\tilde{U}$ is the restriction of the unitary $U$ to the subspace $H^2(\mathbb G)$. Hence, we shall not write $\tilde{U}$ any more.
 Whenever we mention $U$, it will be clear from the context whether it is being applied on $L^2(b \Gamma)$ or on $H^2(\mathbb G)$. In the latter case, the range is $H^2_{\rm anti} (\mathbb D^2)$.

The internal co-ordinates of the (open or closed) symmetrized bidisc will be denoted by $(s, p)$. Several criteria for a member $(s,p)$ of $\mathbb C^2$
to belong to $\mathbb G$ (or $\Gamma$) are known, the interested reader may see Theorem 1.1 in \cite{BhPSR}. Let
$$L^\infty(b\Gamma)=\{\varphi:b\Gamma \to \mathbb C : \text{ there exists } M > 0, \text{ such that } |\varphi(s,p)| \leq M \text{ a.e. in } b\Gamma \}.$$

\begin{definition}
For a function $\varphi$ in $L^\infty(b\Gamma)$, the multiplication operator $M_\varphi$ is defined to be the operator on $L^2(b \Gamma)$:
$$ M_\varphi f (s,p) = \varphi(s,p)f(s,p), $$
for all $f$ in $L^2(b \Gamma)$. The $M_\varphi$ is called the {\em{Laurent operator}} with symbol $\varphi$. The compression of $M_\varphi$ to $H^2(\mathbb G)$ is called {\em{Toeplitz operator}} and is denoted by $T_\varphi$. Therefore
$$
T_\varphi f = Pr M_\varphi f \;\text{  for all $f$ in $H^2( \mathbb G)$.}
$$
\end{definition}

We note that the co-ordinate multiplication operators $M_s$ and $M_p$ are commuting normal operators on $L^2(b \Gamma)$.
We now describe an equivalent way of studying Laurent operators and Toeplitz operators on the symmetrized bidisc. Let $L^\infty_{\text{sym}}(\mathbb T^2)$ denote the sub-algebra of $L^\infty(\mathbb T^2)$ consisting of symmetric functions, i,e., $f(\zeta_1,\zeta_2)=f(\zeta_2,\zeta_1)$ a.e. and $\Pi_1:L^\infty(b\Gamma)\to L^\infty_{\text{sym}}(\mathbb T^2)$ be the $*$-isomorphism defined by
$$
\varphi \mapsto \varphi \circ \pi
$$
where $\pi$ is as defined in (\ref{pi}). Let $\Pi_2:\mathcal B(L^2(b\Gamma))\to \mathcal B(L^2_{\text{anti}}(\mathbb T^2))$ denote the conjugation map by the unitary ${U}$ as defined in (\ref{utilde}), i.e., $$T\mapsto {U} T {U}^*.$$
\begin{proposition}\label{equiv-Laurent}
Let $\Pi_1$ and $\Pi_2$ be the above $*$-isomorphisms. Then the following diagram is commutative:
$$
\begin{CD}
L^\infty(b\Gamma) @> \Pi_1>>  L^\infty_{\text{sym}}(\mathbb T^2)\\
@Vi_1 VV @VVi_2 V\\
\mathcal B(L^2(b\Gamma))@>>\Pi_2> \mathcal B(L^2_{\text{anti}}(\mathbb T^2))
\end{CD},
$$where $i_1$ and $i_2$ are the canonical inclusion maps. Equivalently, for $\varphi \in L^\infty(b\Gamma)$, the operators $M_\varphi$ on $L^2(b\Gamma)$ and $M_{\varphi\circ\pi}$ on $L^2_{\text{anti}}(\mathbb T^2)$ are unitarily equivalent via the unitary $ U$.
\end{proposition}
\begin{proof}
To show that the above diagram commutes all we need to show is that $U M_\varphi  U^*=M_{\varphi \circ \pi}$, for every $\varphi$ in $L^\infty(b\Gamma)$. This follows from the following computation: for every $\varphi$ in $L^\infty(b\Gamma)$ and $f \in L^2_{\text{anti}}(\mathbb T^2)$,
$$
 U M_\varphi  U^*(f)= U (\varphi U^* f)=( \varphi\circ\pi) \frac{1}{\|J\|}J( U^*f\circ\pi )= M_{\varphi\circ \pi}(f).
$$
\end{proof}
As a consequence of the above, given a Toeplitz operator on the Hardy space of the symmetrized bidisc, there is a unitarily equivalent copy of it on $H^2_{\text{anti}}(\mathbb D^2)$.
\begin{corollary}\label{equiv-Toep}
For  $\varphi \in L^\infty(b\Gamma)$, $T_\varphi$ is unitarily equivalent to $T_{\varphi\circ \pi}:=P_aM_{\varphi \circ \pi}|_{H^2_{\text{anti}}(\mathbb D^2)}$, where $P_a$ stands for the projection of $L^2_{\text{anti}}(\mathbb T^2)$ onto $H^2_{\text{anti}}(\mathbb D^2)$.
\end{corollary}
\begin{proof}
This follows from the fact that the operators $M_\varphi$ and $M_{\varphi \circ \pi}$ are unitarily equivalent via the unitary $\tilde U$, which takes $H^2(\mathbb G)$ onto $H^2_{\text{anti}}(\mathbb D^2)$.
\end{proof}

In what follows, the pair $(T_s, T_p)$ will be specially useful, where $T_s f =  M_s f$ and  $T_p f =  M_p f$ for $f$ in $H^2( \mathbb G)$ (no projection is required because $H^2( \mathbb G)$
is invariant under $M_s$ and $M_p$). The unitary $U$ mentioned in the theorem above intertwines $T_s$ with $T_{z_1 + z_2}=M_{z_1 + z_2}|_{H^2_{\text{anti}}(\mathbb T^2)}$ and $T_p$ with $T_{z_1z_2}=M_{z_1z_2}|_{H^2_{\text{anti}}(\mathbb T^2)}$. In the decomposition $L^2(b\Gamma) = H^2( \mathbb G) \oplus (L^2(b\Gamma) \ominus H^2( \mathbb G))$, we have
$$ M_s = \left(
                 \begin{array}{cc}
                   T_s & * \\
                   0 & * \\
                 \end{array}
               \right) \mbox{ and } M_p = \left(
                 \begin{array}{cc}
                   T_p & * \\
                   0 & * \\
                 \end{array}
               \right). $$

 \begin{lemma} The pair $(M_s, M_p)$ is a commuting pair of normal operators and $\sigma(M_s, M_p) = b\Gamma$. \end{lemma}

\begin{proof}
The Laurent operators $M_s$ and $M_p$  are co-ordinate multiplications on $L^2(b\Gamma)$. Hence they are normal and $\sigma(M_s, M_p) = b\Gamma$.
\end{proof}

If we appeal to Proposition \ref{G-unitary}, we see that the pair $(M_s, M_p)$ is a $\Gamma$-unitary. Thus, by Proposition \ref{G-isometry}, $(T_s, T_p)$ is a $\Gamma$-isometry. Since the adjoint pair of a $\Gamma$-contraction is again a $\Gamma$-contraction, the pair $(T_s^*, T_p^*)$ is a $\Gamma$-contraction. So, it has a fundamental operator.

Since polynomials of the form $z_1^j - z_2^j$ with $j=1,2, \ldots$ form a basis for $H^2_{\text{anti}} (\mathbb D^2)$, we define $X$ in $\mathcal B(H^2_{\text{anti}} (\mathbb D^2))$ by defining it on these elements of $H^2_{\text{anti}}(\mathbb D^2)$ and extending linearly:
\begin{equation}
 \label{the-other-shift} X (z_1z_2)^i (z_1^j - z_2^j ) = (z_1z_2)^i (z_1^{j+1} - z_2^{j+1})  \text{ for } i=0,1, \ldots \text{ and } j=1,2, \ldots .
 \end{equation}
Let us denote
\begin{align}\label{InflateFund}
Y := U^* XU.
 \end{align}Since $X$ commutes with $M_{z_1 z_2}|_{H^2_{\rm anti} (\mathbb D^2)}$, $Y$ commutes with $T_p$.

There is a reducing subspace of $X$ that plays a special role. Define
$$\mathcal E = \overline{\operatorname{span}}\{ z_1^j-z_2^j: 1\leq j < \infty \} \subset H^2_{\text{anti}}(\mathbb D^2)$$
and it can be easily checked that $\mathcal E$ is a reducing subspace for $X$. Let $X_0 = X|_\cle$. Consider four Hilbert modules as follows.

\begin{align*}
HM_1: & \;\; H^2(\mathbb G) \text{ with the module action } \xi \cdot h = \xi(T_p, Y)h, \\
HM_2: & \;\; H^2_{\text{anti}}(\mathbb D^2) \text{ with the module action } \xi \cdot h = \xi(M_{z_1 z_2}|_{H^2_{\rm anti} (\mathbb D^2)}, X)h, \\
HM_3: & \;\; H^2(\mathbb D) \otimes \cle \text{ with the module action } \xi \cdot h = \xi(T_z \otimes I_\cle, I_{H^2(\mathbb D)} \otimes X_0)h, \\
HM_4: & \;\; H^2(\mathbb D^2) \text{ with the module action } \xi \cdot h = \xi(T_{z_1}, T_{z_2})h.
\end{align*}

Two Hilbert modules $\clh_1$ and $\clh_2$ over the polynomial ring $\mathbb C[z_1, z_2]$ are said to be $isomorphic$ if there is a unitary $\digamma : \clh_1 \rightarrow \clh_2$ such that
$$\digamma(\xi \cdot h) = \xi \cdot \digamma(h) \text{ for all polynomials } \xi \text{ and all } h \text{ in } \clh_1.$$

\begin{theorem}
The four $\mathbb D^2$-contractive Hilbert modules above are isomorphic, i.e.,

$$ HM_1 \cong  HM_2 \cong  HM_3 \cong HM_4.$$

 \end{theorem}

 \begin{proof}
 The first isomorphism is by virtue of $U$ of \eqref{unotilde}.

 For the second one, note that the vectors $\{ z^i \otimes (z_1^j-z_2^j) : i=0, 1, 2, \ldots $ and $j=1,2,3, \ldots \}$
 form an orthogonal basis for $H^2(\mathbb D) \otimes \cle$. On the other hand, the space $H^2_{\text{anti}}(\mathbb D^2)$ is spanned by the orthogonal set
$\{(z_1z_2)^i(z_1^j-z_2^j): i\geq 0 \text{ and } j\geq 1\}$.  Define the unitary operator from $H^2_{\text{anti}}(\mathbb D^2)$ onto $H^2_{\mathcal E}(\mathbb D)$ by mapping
$$
(z_1z_2)^i(z_1^j-z_2^j) \mapsto z^i\otimes (z_1^j-z_2^j)
$$
and then extending linearly. This preserves norms, is surjective and intertwines $T_{z_1z_2}$ with $T_z \otimes I$ and $X$ with $I \otimes X_0$.

And for the third one, consider the map
\begin{align*}
z^i\otimes (z_1^j-z_2^j)\mapsto \sqrt{2} z_1^iz_2^{j-1} \text{ for }i\geq0,\; j\geq 1,
\end{align*}and extend linearly. This norm-preserving map takes orthonormal basis of $H^2_{\cl{E}}(\mathbb{D})$ to that of $H^2(\mathbb{D}^2)$ and hence is unitary. Also it is easy to see that this unitary map intertwines the operators $T_z$ and $I\otimes X_0$ acting on $H^2_{\cl{E}}(\mathbb{D})$ with the operators $T_{z_1}$ and $T_{z_2}$ acting on $H^2(\mathbb{D}^2)$, respectively. This completes the proof of the theorem.
\end{proof}

The operator $Y$ defined above is important for this note. It will appear again. So, we end this section relating it to the fundamental operator of $(T_s^*,T_p^*)$. The fundamental operator of the adjoint of a $\Gamma$-isometry is especially nice. Indeed, if $(T,V)$ is a $\Gamma$-isometry, then by general theory, delineated at the end of the Preliminaries section, $T^* - TV^*$ is non-zero only on the subspace $\mathcal D_{V^*}$. Moreover, since $V$ is an isometry and hence $V^*D_{V^*} = 0$, we have $T^* - TV^*$ acting on $\mathcal D_{V^*}$ is just $T^*|_{\mathcal D_{V^*}}$. Applying this to the $\Gamma$-isometry $(M_{z_1+z_2}, M_{z_1z_2})|_{H^2_{\rm anti} (\mathbb D^2)}$, a little computation shows that the fundamental operator of the adjoint of $(M_{z_1+z_2}, M_{z_1z_2})|_{H^2_{\rm anti} (\mathbb D^2)}$ is $X_0$. Recall that $\mathcal E$ is a reducing subspace for $X$. By the theorem above, $\mathcal D_{T_p^*}$ is then a reducing subspace for $Y$. By unitary equivalence, the fundamental operator of $(T_s^*,T_p^*)$ is $Y^*|_{\mathcal D_{T_p^*}}$. Therefore, $Y$ is the inflation of the adjoint of the fundamental operator of $(T_s^*,T_p^*)$.

\section{The Brown Halmos relations}

The definition of a Toeplitz operator is analytic. Hence, it is interesting to characterize it algebraically. This is what we do in Theorem II below.

If $M$ is a bounded operator on $L^2(\mathbb T)$ belonging to $\{M_z\}'$, the commutant of the operator $M_z$ on $L^2(\mathbb T)$,
then it is well known that there exists a function $\varphi\in L^\infty( \mathbb T)$ such that $M=M_\varphi$. The following result is an analogue of this phenomenon for the symmetrized bidisc.

\begin{lemma}\label{lauraoperator}
Let $M$ be a bounded operator on $L^2(b \Gamma)$ which commutes with both $M_s$ and $M_p$. Then there exists a function $\varphi\in L^\infty(b \Gamma)$
such that $M=M_\varphi$.
\end{lemma}
\begin{proof}
Since $(M_s,M_p)$ is a pair of commuting normal operators and $\sigma(M_s,M_p)=b\Gamma$, then by the spectral theorem for commuting normal operators the von Neumann algebra generated by $\{M_s,M_p\}$ is $L^\infty(b\Gamma)$, which is a maximal abelian von Neumann algebra. This completes the proof.
\end{proof}

By Proposition \ref{equiv-Laurent}, the above result can be rephrased in the bidisc set up.
\begin{corollary}\label{bidisc-lauraoperator}
Let $M_{z_1+z_2}$ and $M_{z_1z_2}$ denote the multiplication operators on $L^2_{\text{anti}}(\mathbb T^2)$. Then any bounded operator $M$ on $L^2_{\text{anti}}(\mathbb T^2)$ that commutes with both $M_{z_1+z_2}$ and $M_{z_1z_2}$ is of the form $M_\varphi$, for some function $\varphi\in L^\infty_{\text{sym}}(\mathbb T^2)$.
\end{corollary}

\begin{lemma} \label{TsTpPure}
The pair $(T_s, T_p)$ is a pure $\Gamma$-isometry with $(M_s, M_p)$ as its minimal $\Gamma$-unitary extension and $\sigma(T_s, T_p) = \Gamma$. \end{lemma}

\begin{proof} we have already seen that the pair $(T_s, T_p)$ is a $\Gamma$-isometry. The operator $T_p$ is pure because by Corollary (\ref{equiv-Toep}) $T_p$ is unitarily equivalent to $M_{z_1z_2}|_{H^2_{\rm anti} (\mathbb D^2)}$, which is pure. The extension $(M_s,M_p)$ is minimal because $M_{z_1z_2}$ is the minimal unitary extension of $M_{z_1z_2}|_{H^2_{\rm anti} (\mathbb D^2)}$.

It remains to prove that $\sigma(T_s, T_p) = \Gamma$. This is easily accomplished by noting that $H^2(\mathbb G)$ is a reproducing kernel Hilbert space, see page 513 of \cite{bhat-me realization}. Its kernel is
$$ k_S((s_1, p_1), (s_2, p_2)) = \frac{1}{(1 - p_1 \bar{p}_2)^2 - (s_1 - \bar{s}_2 p_1)(\bar{s}_2 - s_1 \bar{p}_2)}.$$
If $(s,p)$ is a point of $\mathbb G$, then $(\sbar, \pbar)$ is a joint eigenvalue of $(T_s^*, T_p^*)$ with the eigenvector $k(\cdot, (s,p))$. Since $(s,p)$ is in $\mathbb G$ if and only if $(\sbar , \pbar)$ is in $\mathbb G$, we have entire $\mathbb G$ in the joint point spectrum of $(T_s^*, T_p^*)$. Since the spectrum is a closed set, $\sigma(T_s, T_p) = \sigma(T_s^*, T_p^*) = \Gamma$.
\end{proof}

We progress with basic properties of Toeplitz operators. Although, a Toeplitz operator is defined in terms of an $L^\infty$ function, it is a difficult question of how to recognize a given bounded operator $T$ on the relevant Hilbert space as a Toeplitz operator. This question was answered for the Hardy space of the unit disc by Brown and Halmos in Theorem 6 of \cite{BH} where they showed that $T$ has to be invariant under conjugation by the unilateral shift. We show that in our context one needs both $T_s$ and $T_p$ to give such a characterization.

\begin{definition}
Let $T$ be a bounded operator on $H^2(\mathbb G)$. We say that $T$ satisfies the Brown-Halmos relations with respect to the $\Gamma$-isometry $(T_s, T_p)$ if
\begin{eqnarray}\label{Toeplitzcharc}
 T_s^*TT_p=TT_s \text{ and } T_p^*TT_p=T.
\end{eqnarray} \end{definition}
It is a natural question whether any of the two Brown-Halmos relations implies the other. We give here an example of an operator $Y$ which satisfies the second one, but not the first.
\begin{example} \label{FO}
This example shows that the operator $ Y $ defined in \eqref{InflateFund} does not satisfy the first of the Brown-Halmos relations.  To that end, we note that
$$ T_s^*YT_p=T_s^*T_pY=T_sY $$
so that the question boils down to whether $Y$ commutes with $T_s$. This is easy to resolve using the $U$ of \eqref{unotilde} because
$$
YT_s(1) = U^* X U T_s(1) =  \frac{1}{\| J \|} U^*   X (z_1^2 - z_2^2) = \frac{1}{\| J \|} U^* (z_1^3 - z_2^3) = s^2-p$$
and
$$T_s Y(1) = T_s U^* X U (1) = \frac{1}{\| J \|} T_s U^*   X (z_1 - z_2) = \frac{1}{\| J \|}T_s U^* (z_1^2 - z_2^2) = T_s s = s^2.$$
However, the second Brown-Halmos relation is satisfied because of commutativity of $Y$ with $T_p$.

\end{example}
\begin{theorem}\label{toepcharcthm}
A Toeplitz operator satisfies the Brown-Halmos relations and vice versa.

\end{theorem}
\begin{proof}
We first prove that the condition is necessary. Let $T$ be a Toeplitz operator with symbol $\varphi$. Then for $f,g\in H^2(\mathbb G)$,
\begin{eqnarray*}
\langle T_p^*T_\varphi T_p f,g \rangle & = & \langle T_\varphi T_p f,T_p g \rangle \\
& = & \langle Pr M_\varphi T_p f, T_p g \rangle \\
&=& \langle M_\varphi M_p f, M_p g \rangle \\
& = &\langle M_\varphi  f,  g \rangle \\
&=& \langle Pr M_\varphi  f,  g \rangle = \langle T_\varphi  f,  g \rangle.
\end{eqnarray*}
Also,
\begin{eqnarray*}
\langle T_s^*T_\varphi T_p  f,  g \rangle_{H^2} &=& \langle Pr M_\varphi T_p  f,  T_s g \rangle_{H^2} \\
& = & \langle M_\varphi M_p  f,  M_s g \rangle_{L^2} \\
& = & \langle M_s^*M_p M_\varphi f, g \rangle_{L^2} \\
&=& \langle M_\varphi M_s f, g \rangle_{L^2} \\
& = & \langle Pr M_\varphi M_s f, g \rangle_{H^2} = \langle T_\varphi T_s f, g \rangle_{H^2}.
\end{eqnarray*}
In the above computation, we have used the equality $M_s=M_s^*M_p$.

Now we prove that the condition is sufficient. To this end we work on $H^2_{\rm anti}(\mathbb D^2)$ and invoke Corollary \ref{equiv-Toep} to draw the conclusion. So let $T$ be a bounded operator on $H^2_{\rm anti}(\mathbb D^2)$ satisfying $ T_{z_1+z_2}^*TT_{z_1z_2}=TT_{z_1+z_2}$ and $T_{z_1z_2}^*TT_{z_1z_2}=T$.  For two different integers $i$ and $j$, let $e_{i,j}:=z_1^iz_2^j-z_1^jz_2^i$. Note that for $n \geq 0$, $M_{z_1z_2}^ne_{i,j}=e_{i+n, j+n}$. Therefore for every different integers $i$ and $j$, there exists a sufficiently large $n$ such that $M_{z_1z_2}^ne_{i,j} \in H^2_{\rm anti}(\mathbb D^2)$. For each $n \geq 0$, define an operator $T_n$ on $L^2_{\text{anti}}(\mathbb T^2)$ by
$$
T_n:={M^{* n}_{z_1z_2}}TP_aM_{z_1z_2}^n,
$$where $P_a$ is the orthogonal projection of $L^2_{\text{anti}}(\mathbb T^2)$ onto $H^2_{\text{anti}}(\mathbb D^2)$.
Let $i,j,k$ and $l$ be integers such that $i \neq j$ and $k \neq l$, then for sufficiently large $n$, we have
\begin{eqnarray}\label{auxtoepcharcone}
\langle T_ne_{i,j}, e_{k,l} \rangle = \langle T M_{z_1z_2}^ne_{i,j}, M_{z_1z_2}^ne_{k,l} \rangle = \langle T e_{i+n,j+n}, e_{k+n,l+n} \rangle.
\end{eqnarray}
Since $T_{z_1z_2}^*TT_{z_1z_2}=T$, we have for every $n\geq 0$, $T_{z_1z_2}^{*n}TT_{z_1z_2}^n=T$. Let $i,j,k$ and $l$ be non-negative integers such that $i \neq j$ and $k \neq l$, then for every $n\geq 0$,
\begin{eqnarray}\label{auxtoepcharctwo}
\langle Te_{i,j}, e_{k,l} \rangle = \langle T T_{z_1z_2}^ne_{i,j}, T_{z_1z_2}^n e_{k,l} \rangle=\langle T e_{i+n,j+n}, e_{k+n,l+n} \rangle.
\end{eqnarray}
Since $\{e_{i,j}: i\neq j \in \mathbb{Z}\}$ is an orthogonal basis for $L^2_{\rm anti}(\mathbb T^2)$ and the sequence of operators $T_n$ on $L^2_{\rm anti}(\mathbb T^2)$ is uniformly bounded by $\|T\|$, by (\ref{auxtoepcharcone}) and $(\ref{auxtoepcharctwo})$ the sequence $T_n$ converges weakly to some operator $T_\infty$ (say) acting on $L^2_{\rm anti}(\mathbb T^2)$.

We now use Corollary \ref{bidisc-lauraoperator} to conclude that $T_\infty=M_\varphi$, for some $\varphi \in L^\infty_{\text{sym}}(\mathbb T^2)$. Therefore we have to show that $T_\infty$ commutes with both $M_{z_1+z_2}$ and $M_{z_1z_2}$. That $T_\infty$ commutes with $M_{z_1z_2}$ is clear from the definition of $T_\infty$. The following computation shows that $T_\infty$ commutes with $M_{z_1z_2}$ also. Let $i,j,k$ and $l$ be integers such that $i \neq j$ and $k \neq l$. Then
\begin{eqnarray*}
&&\langle M_{z_1+z_2}^*T_\infty^*e_{i,j}, e_{k,l} \rangle
\\
&=& \lim_{n}\langle M_{z_1+z_2}^*M_{z_1z_2}^{* n} T^*P_aM_{z_1z_2}^ne_{i,j}, e_{k,l}\rangle
\\
&=& \lim_{n}\langle T_{z_1+z_2}^*T^*M_{z_1z_2}^ne_{i,j},M_{z_1z_2}^ne_{k,l} \rangle \;\;(\text{for sufficiently large $n$})
\\
&=& \lim_{n}\langle T_{z_1z_2}^*T^*T_{z_1+z_2} M_{z_1z_2}^ne_{i,j},M_{z_1z_2}^ne_{k,l} \rangle \;\;(\text{applying (\ref{Toeplitzcharc})})
\\
&=& \lim_{n}\langle M_{z_1z_2}^{* n+1}T^*P_aM_{z_1z_2}^{n+1}M^*_{z_1z_2}M_{z_1+z_2} e_{i,j},e_{k,l} \rangle
\\
&=& \lim_{n}\langle M_{z_1z_2}^{* n+1}P_aT^*P_aM_{z_1z_2}^{n+1}M_{z_1+z_2}^* e_{i,j},e_{k,l} \rangle \;\;(\text{since $M_{z_1+z_2}=M_{z_1+z_2}^*M_{z_1z_2}$})
\\
&=& \langle T_\infty^*M_{z_1+z_2}^* e_{i,j},e_{k,l} \rangle.
\end{eqnarray*}
Therefore there exists a $\varphi\in L^\infty_{\text{sym}}(\mathbb T^2)$ such that $T_\infty=M_\varphi$. Now for $f$ and $g$ in $H^2_{\text{anti}}(\mathbb D^2)$, we have
\begin{eqnarray*}
\langle P_aM_\varphi f,g \rangle = \langle M_\varphi f,g \rangle & = & \langle T_\infty f,g \rangle \\
&=& \lim_{n}\langle T_n f,g \rangle =\lim_{n} \langle T  T_{z_1z_2}^n f,T_{z_1z_2}^ng \rangle =  \langle Tf,g \rangle.
\end{eqnarray*}
Hence $T$ is the Toeplitz operator with symbol $\varphi$.
\end{proof}
The following is a straightforward consequence of the characterization of Toeplitz operators obtained above.
\begin{corollary} \label{CommutesHenceToeplitz}
If $T$ is a bounded operator on $H^2(\mathbb G)$ that commutes with both $T_s$ and $T_p$, then $T$ satisfies the Brown-Halmos relations and hence is a Toeplitz operator.
\end{corollary}
\begin{proof}
It is given that $TT_p = T_pT$. Multiplying both sides from the left by $T_p^*$, we get that $T_p^*TT_p=T$ because $T_p$ is an isometry. The following simple computation shows that $T$ also satisfies the other relation.
$$
T_s^*TT_p=T_s^*T_pT=T_sT=TT_s,
$$where we used the fact that $(T_s,T_p)$ is a $\Gamma$-isometry and hence $T_s=T_s^*T_p$.
\end{proof}

\section{Further properties of a Toeplitz operator}
In this section, we study further properties of Toeplitz operators and characterize Toeplitz operators with analytic and co-analytic symbols.
\begin{lemma}
For $\varphi\in L^{\infty}(b\Gamma)$ if $T_\varphi$ is the zero operator, then $\varphi=0$, a.e. In other words, the map $\varphi \mapsto T_\varphi$ from $L^\infty (b\Gamma)$
into the set of all Toeplitz operators on the symmetrized bidisc, is injective.
\end{lemma}
\begin{proof}
Let $\varphi\circ\pi(z_1,z_2)=\sum_{i,j\in\mathbb Z}\alpha_{i,j}z_1^iz_2^j\in L^\infty_{\text{sym}}(\mathbb T^2)$.
Then $T_{\varphi\circ\pi}$ on $H^2_{\text{anti}}(\mathbb D^2)$ is the zero operator. Now we have for every $m,k\geq 0$ and $n,l \geq 1$,
\begin{eqnarray*}
0&=&\langle T_{\varphi\circ \pi} (z_1z_2)^m(z_1^n-z_2^n), (z_1z_2)^k(z_1^l-z_2^l) \rangle
\\&=&\langle \sum_{i,j\in\mathbb Z}\alpha_{i,j}(z_1^{i+m+n}z_2^{j+m} -z_1^{i+m}z_2^{j+m+n}), (z_1z_2)^k(z_1^l-z_2^l)\rangle\\
&=&\alpha_{k+l-m-n,k-m}+\alpha_{k-m,k+l-m-n}-\alpha_{k+l-m,k-m-n}-\alpha_{k-m-n,k+l-m}\\
&=&2(\alpha_{k+l-m-n,k-m}-\alpha_{k-m-n,k+l-m}).
\end{eqnarray*}
To obtain the last equality we have used the fact that $\alpha_{i,j}=\alpha_{j,i}$ for all $i,j\in\mathbb{Z}$.
Since the sequence $\{\alpha_{i,j}\}$ is square summable, the above computation says that for every $m,k\geq 0$ and $n,l \geq 1$,
$$\alpha_{k-m-n+l,k-m}=\alpha_{k-m-n,k-m+l}=0.$$
Note that $\{k-m: m,k\ge 0\}=\mathbb{Z}$ and for fixed $k,m\ge 0$, $\{(k-m)-(n-l): n,l\ge 1\}=\mathbb{Z}$.
Hence $\alpha_{i,j}=0$ for all $i,j\in\mathbb{Z}$. This completes the proof.
\end{proof}
It is easy to see that the space $H^\infty(\mathbb G)$ consisting of all bounded analytic functions on $\mathbb G$ is contained in $H^2(\mathbb G)$. We identify $H^\infty(\mathbb G)$ with its boundary functions. In other words,
$$
H^\infty(\mathbb G)=\{\varphi\in L^\infty(b\Gamma): \varphi\circ\pi \text{ has no non-zero negative Fourier coefficients}\}.
$$
\begin{definition}
A Toeplitz operator $T_\varphi$ is called
\begin{enumerate}
\item an analytic Toeplitz operator if $\varphi$ is in $H^\infty(\mathbb G)$,
\item a co-analytic Toeplitz operator if $T_\varphi^*$ is an analytic Toeplitz operator.
\end{enumerate}
\end{definition}
Our next goal is to characterize analytic Toeplitz operators. But to be able to do that we need to define the following notion and prove the proposition following it.
\begin{definition}
Let $\varphi$ be in $L^\infty(b\Gamma)$. The operator $H_\varphi:H^2(\mathbb G)\to L^2(b\Gamma)\ominus H^2(\mathbb G)$ defined by
$$
H_\varphi f = (I -Pr)M_\varphi f
$$for all $f\in H^2(\mathbb G)$, is called a Hankel operator.
\end{definition}
We write down a few observations about Toeplitz operators for the sake of completeness. The proofs are similar to the one dimensional case.

\begin{proposition} \label{EasyObs} Let $\varphi, \psi \in L^\infty(b\Gamma)$. Then
\begin{enumerate}
\item $T_\varphi^* = T_{\overline{\varphi}}$.
\item The product $T_\varphi T_\psi$ is a Toeplitz operator if $\overline{\varphi}$ or $\psi$ is analytic. In each case, $T_\varphi T_\psi = T_{\varphi \psi}$.
\item  $T_\varphi T_\psi - T_{\varphi \psi}=-H_{\overline{\varphi}}^*H_\psi.$
\item For an operator $T$, let $\Pi(T)$ be the approximate point spectrum of $T$. Then
$$ \mbox{essential range of } \varphi=\Pi(M_\varphi)=\sigma(M_\varphi)\subseteq \Pi(T_\varphi)\subseteq \sigma(T_\varphi).$$ Hence
\begin{enumerate}
\item $\|\varphi\|_\infty = \|M_\varphi\|=\|T_\varphi\|=r(T_\varphi)$ and \item $\|T_\varphi - K\| \geq \|T_\varphi\|$, for every compact operator $K$ on $H^2(\mathbb G)$.
\end{enumerate}
\end{enumerate} \end{proposition}

Now we are ready to characterize Toeplitz operators with analytic symbol.
\begin{theorem}
Let $T_\varphi$ be a Toeplitz operator. Then the following are equivalent:
\begin{enumerate}
\item[(i)] $T_\varphi$ is an analytic Toeplitz operator;
\item[(ii)] $T_\varphi$ commutes with $T_p$;
\item[(iii)]$T_\varphi(Ran T_p)\subseteq Ran T_p$;
\item[(iv)] $T_pT_\varphi$ is a Toeplitz operator;
\item[(v)]  $T_\varphi$ commutes with $T_s$;
\item[(vi)] $T_sT_\varphi$ is a Toeplitz operator.
\end{enumerate}
\end{theorem}
\begin{proof}
\underline{$(i)\Leftrightarrow(ii)$:}That $(i)\Rightarrow (ii)$ is easy. To prove the other direction, we use part $(3)$ of Proposition \ref{EasyObs} to get that $H_{\overline{p}}^*H_\varphi=0$. This shows that the corresponding product of Hankel operators on $H^2_{\text{anti}}(\mathbb D^2)$ is also zero, that is
$H_{\overline{z_1z_2}}^*H_{\varphi\circ\pi}=0$. Let the power series expansion of $\varphi\circ\pi \in L^\infty_\text{symm}(\mathbb T^2)$ be
$$\varphi\circ\pi(z_1,z_2)=\sum_{m,n \in \mathbb Z}\alpha_{m,n}z_1^mz_2^n \text{ for all $z_1,z_2 \in \mathbb T$}.$$ Since $\varphi\circ\pi$ is symmetric we have $\alpha_{m,n}=\alpha_{n,m}$, for every $m,n\in \mathbb Z$. For $k,r \geq 0$ and $l \geq 1$, we have
\begin{eqnarray*}
0&=&\langle H_{\varphi\circ\pi} (z_1z_2)^r(z_1^l-z_2^l), H_{\overline{z_1z_2}}(z_1^{k+1}-z_2^{k+1}) \rangle_{L^2(\mathbb T^2)}
\\
&=& \langle \sum_{m,n \in \mathbb Z}\alpha_{m,n}z_1^mz_2^n(z_1z_2)^r(z_1^l-z_2^l), (z_1^{k}\overline{z_2}-\overline{z_1}z_2^{k})\rangle_{L^2(\mathbb T^2)}
\\
&=& \langle \sum_{m,n \in \mathbb Z}\alpha_{m,n}(z_1^{m+r+l}z_2^{n+r}-z_1^{m+r}z_2^{n+r+l}),(z_1^{k}\overline{z_2}-\overline{z_1}z_2^{k})\rangle_{L^2(\mathbb T^2)}
\\
&=&2(\alpha_{k-r-l,-r-1}-\alpha_{k-r,-r-1-l}),
\end{eqnarray*}
where to obtain the last equality we have used $\alpha_{m,n}=\alpha_{n,m}$ for every $m,n \in \mathbb Z$. Now since the sequence $\{\alpha_{m,n}\}$ is square summable, we conclude that for  every $k,r \geq 0$ and $l \geq 1$
$$
\alpha_{-r-1,(k-l)-r}=\alpha_{-(r+l)-1,k-r}=0.
$$From these equalities we claim that $\alpha_{m,n}=0$, unless both of $m$ and $n$ are non-negative, which would imply that $\varphi$ is analytic. First we show that if $m\geq 0$ and $n\geq 1$, then $\alpha_{-n,m}=\alpha_{m,-n}=0$. For that we choose $r=n-1$ and $k,l$ such that $k-l=m+n-1$. For this choice of $k,r$ and $l$ we have $0=\alpha_{-r-1,(k-l)-r}=\alpha_{-n,m}$. Now we show that if $m\geq 1$ and $n\geq 0$, then $\alpha_{-m,-n}=0$. To this end, we choose $r=m-1$ and $k,l$ such that $k-l=m-n-1$. For this choice of $k,r$ and $l$ we have $0=\alpha_{-r-1,(k-l)-r}=\alpha_{-m,-n}$.

\underline{$(ii)\Leftrightarrow (iii)$:} The part $(ii)\Rightarrow(iii)$ is easy. Conversely, suppose that $RanT_p$ is invariant under $T_\varphi$. Since $RanT_p$ is closed, we have for every $f\in H^2(\mathbb G)$,
\begin{eqnarray*}
&&T_\varphi T_p f=T_p g_f \text{ for some $g_f$ in $H^2(\mathbb G)$.}
\\
&\Rightarrow& T_p^*T_\varphi T_p f=g_f \Rightarrow T_\varphi f =g_f \;(\text{by Theorem \ref{toepcharcthm}}).
\end{eqnarray*} Hence $T_\varphi T_p=T_p T_\varphi$.

\underline{$(ii)\Leftrightarrow (iv)$:} If $T_\varphi$ commutes with $T_p$, then $T_pT_\varphi$ is same as $T_\varphi T_p$, which is a Toeplitz operator by Proposition \ref{EasyObs}. Conversely, if $T_pT_\varphi$ is a Toeplitz operator, then it satisfies Brown-Halmos relations, the second one of which implies that $T_\varphi$ commutes with $T_p$.

\underline{$(i)\Leftrightarrow(v)$:} For an analytic symbol $\varphi$, $T_\varphi$ obviously commutes with $T_s$. The proof of the converse direction is done by the same technique as in the proof of $(ii)\Rightarrow (i)$. If $T_\varphi$ commutes with $T_s$, then by part $(3)$ of Proposition \ref{EasyObs} we have $H_{\overline{s}}^*H_\varphi=0$. Suppose $\varphi\circ\pi \in L^\infty_\text{symm}(\mathbb T^2)$ has the following power series expansion
$$\varphi\circ\pi(z_1,z_2)=\sum_{m,n \in \mathbb Z}\alpha_{m,n}z_1^mz_2^n \text{ for all $z_1,z_2 \in \mathbb T$}.$$ For every $k,l \geq 1$ and $r \geq 0$, we have
\begin{eqnarray*}
0&=&\langle H_{\varphi\circ\pi} (z_1z_2)^r(z_1^l-z_2^l), H_{\overline{z_1+z_2}}(z_1^{k}-z_2^{k}) \rangle_{L^2(\mathbb T^2)}
\\
&=& \langle \sum_{m,n \in \mathbb Z}\alpha_{m,n}z_1^mz_2^n(z_1z_2)^r(z_1^l-z_2^l), (z_1^{k}\overline{z_2}-\overline{z_1}z_2^{k})\rangle_{L^2(\mathbb T^2)}
\\
&=& \langle \sum_{m,n \in \mathbb Z}\alpha_{m,n}(z_1^{m+r+l}z_2^{n+r}-z_1^{m+r}z_2^{n+r+l}),(z_1^{k}\overline{z_2}-\overline{z_1}z_2^{k})\rangle_{L^2(\mathbb T^2)}
\\
&=&2(\alpha_{-r-1,(k-l)-r}-\alpha_{-(r+l)-1,k-r}).
\end{eqnarray*}
Similar argument as in the proof of $(ii)\Rightarrow (i)$ reveals that $\alpha_{m,n}=0$, if either of $m$ and $n$ is negative, in other words, $\varphi$ is analytic.


\underline{$(v)\Leftrightarrow (vi)$:} The implication $(v)\Rightarrow (vi)$ follows from Proposition \ref{EasyObs}. Conversely suppose that $T_sT_\varphi$ is a Toeplitz operator. Therefore applying Theorem \ref{toepcharcthm} and the relation $T_s={T_s}^*T_p$, we get $T_\varphi T_s={T_s}^*T_\varphi T_p={T_p}^*T_sT_\varphi T_p=T_sT_\varphi.$
\end{proof}
The following is a direct consequence of the preceding theorem.
\begin{corollary}
Let $T_\psi$ be a Toeplitz operator. Then the following are equivalent:
\begin{enumerate}
\item[(i)] $T_\psi$ is a co-analytic Toeplitz operator;
\item[(ii)] $T_\psi$ commutes with $T_p^*$;
\item[(iii)] $T_\psi^*(Ran T_p)\subseteq Ran T_p$;
\item[(iv)]$T_pT_\varphi^*$ is a Toeplitz operator;
\item[(v)]  $T_\varphi^*$ commutes with $T_s$;
\item[(vii)] $T_sT_\varphi^*$ is a Toeplitz operator.
\end{enumerate}
\end{corollary}

We end this section with two facts about Toeplitz operators on the symmetrized bidisc - one is similar to the unit disc and the other is dissimilar.
\begin{proposition}
The only compact Toeplitz operator on the symmetrized bidisc is zero.
\end{proposition}
\begin{proof}
The proof is similar to that in case of the unit disc. Let $T_\varphi$ be a compact Toeplitz operator. For every $m>n\geq0$, let $e_{m,n}=z_1^mz_2^n-z_1^nz_2^m$. Then $\{e_{m,n}:m>n\geq0\}$ is an orthogonal basis of $H^2_{\text{anti}}(\mathbb D^2)$. Since $T_\varphi$ is compact, $\|T_{\varphi\circ\pi} e_{m,n}\|\to 0$ as $m,n\to \infty$. Also $T_{z_1z_2}^*T_{\varphi\circ\pi} T_{z_1z_2}=T_{\varphi\circ\pi}$, so we have for every $r\geq 0$,
$$
|\langle T_{\varphi\circ\pi} e_{m,n}, e_{k,l} \rangle| =|\langle T_{\varphi\circ\pi} e_{m+r,n+r}, e_{k+r,l+r} \rangle|\leq \sqrt{2}\|T_{\varphi\circ\pi} e_{m+r,n+r}\|\to 0$$ as $r\to \infty$, which shows that $T_{\varphi\circ\pi}$ is zero, since $m>n\geq0$ and $k>l\geq0$ are arbitrary.
\end{proof}
It has been observed over the last decade that operator theory on the symmetrized bidisc enjoys some one dimensional phenomena. Specifically, we would like to mention the following peculiar fact related to the minimal normal boundary dilation of a $\Gamma$-contraction $(S,P)$. The space on which the minimal normal boundary dilation of $(S,P)$ acts is the same as the space of minimal unitary dilation of the contraction $P$ (\cite{Bhat-Sau-PRIMS}, Theorem 4 and the discussion preceding it). However, the following example shows that the Coburn Alternative, which has several useful consequences in the study of Toeplitz operators on the unit disc, fails to hold true in the symmetrized bidisc.
\begin{proposition}[The Coburn Alternative]
For a non-zero function $\varphi$ in $L^\infty(\mathbb T)$, either $T_\varphi$ or ${T_\varphi}^*$ is injective.
\end{proposition}
See Theorem 3.3.10 of the book \cite{Mart-Rose} for a proof of this. To show that it fails in the case of the symmetrized bidisc, we choose the symbol to be $\varphi(z_1,z_2)=z_1^2\overline{z_2}^2+\overline{z_1}^2z_2^2.$ Note that $\varphi$ is in $L^\infty_{sym}(\mathbb T^2)$ and $T_\varphi(z_1-z_2)=0=T_\varphi^*(z_1-z_2)$.

\section{Asymptotic Toeplitz operators and Compactness}
The weak limit of a sequence $\{{T_z^*}^nTT_z^n\}_{n\ge 1}$ from $\mathcal B(H^2(\mathbb D))$ is a Toeplitz operator. The second co-ordinate multiplier $T_p$ of $H^2(\mathbb G)$ is unitarily equivalent to $T_z$ on a vector-valued Hardy space on the unit disc. But, we have seen an example which shows that an operator need not be a Toeplitz operator even if it commutes with $T_p$. Therefore, if $T \in \mathcal B(H^2(\mathbb G))$ is such that the sequence $\{{T_p^*}^nTT_p^n\}_{n\ge 1}$ is weakly convergent, the weak limit, $B$ say, may not be a Toeplitz operator on $H^2(\mathbb G)$. The following lemma gives a necessary and sufficient condition for when $B$ is Toeplitz.

\begin{lemma}\label{cond-limit-Toepitz}
Let $T$ and $B$ be bounded operators on $H^2(\mathbb G)$ such that ${T_p^*}^nTT_p^n\to B$ weakly. Then $B$ is a Toeplitz operator if and only if
$$
{T_p^*}^n[T,T_s]T_p^n \to 0 \text{ weakly,}
$$
where $[T,T_s]$ denotes the commutator of $T$ and $T_s$.
\end{lemma}
\begin{proof}
Note that if $T$ and $B$ are bounded operators on $H^2(\mathbb G)$ such that ${T_p^*}^nTT_p^n\to B$ weakly, then $T_p^*BT_p=B$. Suppose ${T_p^*}^n[T,T_s]T_p^n \to 0$ weakly. To prove that $B$ is Toeplitz, it remains to show that $B$ satisfies the first Brown-Halmos relation with respect to the $\Gamma$-isometry $(T_s,T_p)$.
\begin{eqnarray*}
T_s^*BT_p&=&\text{w-}\!\lim T_s^*({T_p^*}^nTT_p^n)T_p\\
& = & \text{w-}\!\lim {T_p^*}^n(T_s^*TT_p)T_p^n \\
& = & \text{w-}\!\lim {T_p^*}^{n+1}T_sTT_p^{n+1}\\
&=&\text{w-}\!\lim {T_p^*}^{n+1}(T_sT-TT_s+TT_s)T_p^{n+1}\\
& = & \text{w-}\!\lim {T_p^*}^{n+1}TT_p^{n+1}T_s=BT_s.
\end{eqnarray*}
Conversely, suppose that the weak limit $B$ of ${T_p^*}^nTT_p^n$ is a Toeplitz operator and hence satisfies the Brown-Halmos relations. Thus,
\begin{eqnarray*}
\text{w-}\!\lim{T_p^*}^n(TT_s-T_sT)T_p^n&=&\text{w-}\!\lim({T_p^*}^nTT_p^nT_s-{T_p^*}^nT_s^*T_pTT_p^n)
\\
&=&\text{w-}\!\lim({T_p^*}^nTT_p^nT_s-T_s^*{T_p^*}^{n-1}TT_p^{n-1}T_p)\\
&=&BT_s-T_s^*BT_p=0.
\end{eqnarray*}
\end{proof}
The next result characterizes compact operators on $H^2(\mathbb G)$.

\begin{proposition}\label{charc-compact}
For every $n\geq1$, let $\eta_n : \mathcal B (H^2(\mathbb G)) \to \mathcal B ( H^2(\mathbb G) \oplus H^2(\mathbb G))$ be the completely positive map defined by
$$ \eta_n (T) := \left(
                  \begin{array}{c}
                    Y^{*n} \\
                    T_p^{*n} \\
                  \end{array}
                \right)T \left(
                           \begin{array}{cc}
                             Y^n, & T_p^n \\
                           \end{array}
                         \right),$$
where $Y$ is the bounded operator on $H^2(\mathbb G)$ as defined in \eqref{InflateFund}. Then $T \in \mathcal B(H^2(\mathbb G))$ is compact if and only if $\eta_n(T)\to 0$ in norm as $n\to \infty$.
\end{proposition}
\begin{proof}
 By virtue of Theorem I, a bounded operator $T$ on $H^2(\mathbb G)$ satisfies the convergence conditions in the statement if and only if the  isomorphic copy $\tilde{T}$ of $T$ on $H^2(\mathbb D^2)$ satisfies ${T_{z_i}^*}^m\tilde T T_{z_j}^m \to 0$ in norm for $1\leq i,j\leq 2$. This is known to be a characterization of compact operators on $H^2(\mathbb D^2)$, see \cite{MSS} for example. That completes the proof.
\end{proof}

\begin{definition}
A bounded operator $T$ on $H^2(\mathbb G)$ is called an asymptotic Toeplitz operator if ${T_p^*}^n[T,T_s]T_p^n\to 0$, ${T_p^*}^nTT_p^n\to B$ and $\eta_n(T-B) \to 0$, where $\eta_n$ is as in Proposition \ref{charc-compact}.
\end{definition}
\begin{theorem}
A bounded operator $T$ on $H^2(\mathbb G)$ is an asymptotic Toeplitz operator if and only if $T$ is the sum of a compact operator and a Toeplitz operator.
\end{theorem}
\begin{proof}
If $T$ is a asymptotic Toeplitz operator and ${T_p^*}^nTT_p^n$ converges to $B$, then it follows from Lemma \ref{cond-limit-Toepitz} that $B$ is a Toeplitz operator because ${T_p^*}^n[T,T_s]T_p^n\to 0$. Also, since $\eta_n(T-B)\to 0$, by Proposition \ref{charc-compact}, $T-B$ is a compact operator. Hence $T$ is the sum of a compact operator and a Toeplitz operator.

Conversely, let $T=K+T_\varphi$, where $K$ is some compact operator. Then by Proposition \ref{charc-compact}, ${T_p^*}^nTT_p^n\to T_\varphi$. Since $T_\varphi$ is Toeplitz, by Lemma \ref{cond-limit-Toepitz}, ${T_p^*}^n[T,T_s]T_p^n\to 0$. And finally, since $K$ is compact, by Proposition \ref{charc-compact}, $\eta_n(T-T_\varphi)\to 0$. Hence $T$ is an asymptotic Toeplitz operator.
\end{proof}
\begin{remark}
If $T$ is an operator such that both $T_p^{*n}TT_p^n$ and $Y^{*n}TY^n$ converge to $T$, even then it is not necessary that $T$ is a Toeplitz operator. For example, choose $T=Y$. Because $Y$ is an isometry and it commutes with $T_p$, for every $n\geq 0$, $Y^{* n}YY^n=Y$ and $T_p^{*n}YT_p^n = Y$. But we have noticed in Example \ref{FO} that $Y$ is not a Toeplitz operator.
\end{remark}

%

\section{A Commutant Lifting Result}

It is a natural generalization of the concept of Toeplitz operators to replace the multiplication by the co-ordinate multiplier by a more general isometry
(in the classical case of Brown and Halmos). Moreover, depending on the domain, one can introduce a tuple of operators with a suitable property.
Prunaru did it for the Euclidean ball $\mathbb B_d$. The natural operator tuple to consider there is a spherical isometry, i.e.,
a commuting tuple $T = (T_1, T_2, \ldots ,T_d)$ of bounded operators with the property $T_1^* T_1 + T_2^*T_2 + \cdots +T_d^* T_d = I$,
its prototypical example being the tuple of co-ordinate multiplications $T_z = (T_{z_1}, T_{z_2}, \ldots ,T_{z_d})$ on the Hardy space of the Euclidean ball.
Prunaru called an operator $X$ a Toeplitz operator with respect to a given spherical isometry $T$ if $T_1^* X T_1 + T_2^* X T_2 + \cdots +T_d^* X T_d = X$.

\begin{definition} Given a Hilbert space $\mathcal H$, a $\Gamma$-isometry $(S,P)$ on $\mathcal H$ and a bounded operator $T$ on $\mathcal H$,
we say that $T$ satisfies the Brown-Halmos relation with respect to the $\Gamma$-isometry  $(S,P)$ (or just satisfies the Brown-Halmos relation when the pair $(S,P)$
is clear from the context) if \begin{equation}\label{GammaT}
 S^*TP=TS \text{ and }P^*TP=T.
\end{equation} \end{definition}
\begin{definition}
\label{taof}
We say that a family $\mathcal F=\{(S_\alpha,P_\alpha):\alpha \in \Lambda\}$ of $\Gamma$-isometries on a Hilbert space $\mathcal H$ is
commuting if the union $\cup_{\alpha\in \Lambda}\{S_\alpha,P_\alpha\}$ is a commutative set of operators.

For a commuting family $\mathcal F$ of $\Gamma$-isometries on a Hilbert space $\mathcal H$, let $\mathcal T(\mathcal F)$ be the set of all operators $X \in \mathcal B(\mathcal H)$ such that
$$
S_\alpha^*XP_\alpha=XS_\alpha \text{ and } P_\alpha^*XP_\alpha=X, \text{ for all $\alpha \in \Lambda$.}
$$
In other words, an element of $\mathcal T(\mathcal F)$ satisfies the Brown-Halmos condition for each $\alpha$.

\end{definition}

\begin{remark}
$\mathcal T(\mathcal F)$ contains $\mathcal F$ and the commutant of $\mathcal F$.  \end{remark}

   The main result of this section is the following. It is similar in spirit to Theorem 1.2 of Prunaru \cite{Prunaru} whose roots can
    be traced back to Section 3 of Beltita and Prunaru \cite{BP}.
    The difference in our theorem lies in the $S_\alpha$. We shall apply Beltita and Prunaru's ideas to obtain simultaneous dilation
    of the $P_\alpha$ and then note how the representation acts on $S_\alpha$. It will be clear in course of  the proof that the dilation space
    is no bigger than that of the simultaneous dilation of  $P_\alpha$.

\begin{theorem}
Let $\mathcal F=\{(S_\alpha,P_\alpha):\alpha \in \Lambda\}$ be a commuting family of $\Gamma$-isometries on a Hilbert space
$\mathcal H$. Then
\begin{enumerate}
\item[(1)] There exists a commuting family $\mathcal G=\{(R_\alpha,U_\alpha):\alpha\in\Lambda)\}$ of $\Gamma$-unitaries acting on a Hilbert space $\mathcal K$ containing $\mathcal H$ such that each pair $(R_\alpha,U_\alpha)$ is an extension of $(S_\alpha,P_\alpha)$. Moreover, $\mathcal G$ is the minimal extension of $\mathcal F$ in the sense that $\mathcal K$ is the smallest reducing subspace of each $R_\alpha$ and $U_\alpha$ containing $\mathcal H$. In fact,
    $$\mathcal K=\{U_{\alpha_1}^{m_1}U_{\alpha_2}^{m_2}\cdots U_{\alpha_n}^{m_n}h:h\in\mathcal H, n\in\mathbb N \text{ and for } 1\leq j \leq n,\alpha_j \in \Lambda \text{ and } m_j \in \mathbb Z\}.$$
    Moreover, any $X \in \mathcal B(\mathcal H)$ commutes with $\mathcal F$ if and only if $X$ has a unique norm preserving extension $Y$ acting on $\mathcal K$ commuting with $\mathcal G$.
\item[(2)] An $X \in \mathcal B(\mathcal H)$ is in $\mathcal T(\mathcal F)$ if and only if there exists an $Y \in \mathcal B(\mathcal K)$ in the commutant of the von-Neumann algebra generated by $\{R_\alpha,U_\alpha:\alpha \in \Lambda\}$ such that
    $
    X=P_\mathcal HY|_{\mathcal H}.
    $
\item[(3)] Let $\mathcal C^*(\mathcal F)$ and $\mathcal C^*(\mathcal G)$ denote the unital $\mathcal C^*$-algebras generated by
$\{S_\alpha,P_\alpha:\alpha \in \Lambda\}$ and $\{R_\alpha,U_\alpha:\alpha \in \Lambda\}$, respectively and $\mathcal I(\mathcal F)$
denote the closed ideal of $\mathcal C^*(\mathcal F)$ generated by all the commutators $XY-YX$ for $X,Y\in \mathcal C^*(\mathcal F)\cap \mathcal T(\mathcal F)$.
Then there exists a short exact sequence
    $$
    0\rightarrow\mathcal I(\mathcal F)\hookrightarrow \mathcal C^*(\mathcal F)\xrightarrow{\pi_0} \mathcal C^*(\mathcal G)\rightarrow 0
    $$
    with a completely isometric cross section, where $\pi_0: \mathcal C^*(\mathcal F)\to \mathcal C^*(\mathcal G)$ is the canonical unital $*$-homomorphism which sends the generating set $\mathcal F$
    to the corresponding generating set $\mathcal G$, i.e., $\pi_0(P_{\alpha})=U_{\alpha}$ and $\pi_0(S_{\alpha})=R_{\alpha}$ for all $\alpha\in\Lambda$.
\end{enumerate}
\end{theorem}

\begin{remark} A commuting family $\mathcal G=\{(R_\alpha,U_\alpha):\alpha\in\Lambda)\}$ of $\Gamma$-unitaries as above is said to $extend$ $\mathcal F$. \end{remark}

\begin{proof}

For each $\alpha \in \Lambda$, define  $\Phi_\alpha:\mathcal B(\mathcal H) \to \mathcal B(\mathcal H)$ by
$$
\Phi_\alpha(X)=P_\alpha^*XP_\alpha.
$$
Then the family $\{\Phi_\alpha\}_{\alpha\in \Lambda}$ consists of commuting, completely positive, unital, normal mappings acting on $\mathcal B( \mathcal H)$.
Therefore, by Lemma 2.3 of \cite{Prunaru}, there exists a completely positive map
$\Phi:\mathcal B(\mathcal H) \to \mathcal B(\mathcal H)$ such that $\Phi \circ\Phi=\Phi$ and
$$
\text{Ran}\Phi=\{X\in \mathcal B(\mathcal H):\Phi_\alpha(X)=P_\alpha^*XP_\alpha=X, \text{ for all $\alpha\in \Lambda$}\}.
$$
In particular, $\Phi(X)=X$ for all $X\in\mathcal T(\mathcal P)$ where
$$\mathcal T (\mathcal P)=\{X:P_\alpha^*XP_\alpha=X, \text{ for all }\alpha \in \Lambda\}.$$
Also since $\Phi$ is an idempotent unital completely positive map, it follows from a well-known result of ~\cite{CE} that
\begin{equation}
\label{ideal}
 \Phi(\Phi(X)Y)=\Phi(X\Phi(Y))=\Phi(\Phi(X)\Phi(Y))
\end{equation}

Let $\mathcal C^*(\mathcal T(\mathcal P))$ denote the $\mathcal C^*$-algebra generated by $\mathcal T(\mathcal P)$ and $\Phi_0$ denote the restriction of $\Phi$ to $\mathcal C^*(\mathcal T(\mathcal P))$.
Consider the minimal Stinespring dilation $\pi:\mathcal C^*(\mathcal T(\mathcal P))\to \mathcal B(\mathcal K)$ of $\Phi_0$. Hence $\Phi_0(X)=V^*\pi(X)V$ for some isometry $V:\mathcal H\to \mathcal K$ and for all $X\in \mathcal B(\mathcal H)$. It follows from ~\eqref{ideal} that $\text{Ker} \Phi_0$ is an ideal of $\mathcal C^*(\mathcal T(\mathcal P))$ and therefore $\text{Ker}\Phi_0=\text{Ker}\pi$ and the mapping $\rho:\pi(\mathcal  C^*(\mathcal T(\mathcal P)))\to \mathcal B(\mathcal H) $ defined by $\rho(\pi(X))= V^*\pi(X)V$ for $X\in \mathcal C^*(\mathcal T(\mathcal P))$ is a complete isometry
such that $\pi\circ\rho= id_{\pi(\mathcal C^*(\mathcal T(\mathcal P)))}$ and $\text{Ran} \rho=\text{Ran}\Phi$.
Define $U_\alpha =  \pi(P_\alpha)$. The two properties below are obtained from the proof of
Theorem 1.2 of Prunaru \cite{Prunaru} applied to $\mathcal P$:
\begin{enumerate}
 \item[($\mathbf{P_1}$)]
 The commuting family of unitaries $\mathcal U=\{U_{\alpha}=\pi(P_\alpha): \alpha\in \Lambda\}$ is a minimal unitary extension of the family of isometries $\mathcal P$, i.e.,
 $$P_{\alpha}=V^*U_{\alpha}V \mbox{ and } U_{\alpha} (V\mathcal H)\subseteq V\mathcal H$$
 for all $\alpha\in\Lambda$ and $\mathcal K$ is the minimal reducing subspace containing $V\mathcal H$ for the family $\mathcal U$.
 \item[($\mathbf{P_2}$)]
If $X\in\mathcal B(\mathcal H)$ belongs to the commutant of $\mathcal P$, then $\widehat{X}=\pi(X)$ is the unique norm preserving extension of $X$ in the commutant of $\mathcal U$ which leaves $V\mathcal H$ invariant.
\end{enumerate}

We identify $\mathcal H$ with $V\mathcal H$ and view $\mathcal H$ as a subspace of $\mathcal K$. Applying ($\mathbf{P_2}$) from above, we get $R_\alpha= \pi(S_{\alpha})$  to be a norm preserving extension of $S_\alpha$. Moreover, $R_\alpha = \pi(S_{\alpha}) = \pi(S_\alpha) = \pi(S_\alpha^*P_\alpha) = \pi(S_\alpha)^* \pi(P_\alpha) = R_\alpha^* U_\alpha $ for each $\alpha\in \Lambda.$ Hence by part (3) of Proposition \ref{G-unitary}, $(R_\alpha,U_\alpha)$ is a $\Gamma$-unitary for each $\alpha$. It is now clear from  property ($\mathbf{P_1}$) that the commuting family of $\Gamma$-unitaries
$\mathcal G=\{(R_\alpha,U_\alpha):\alpha\in\Lambda)\}$ is a minimal normal extension of the commuting family of $\Gamma$-isometries $\mathcal F$.

For the rest of part (1), note that if $X$ commutes with $\mathcal F$, then $X$ belongs to the commutant of $\mathcal P$. Therefore again by property ($\mathbf{P_2}$),
$\pi(X)$ is the unique norm preserving extension of $X$ in the commutant of $\mathcal U$. Moreover, $\pi(X)$ belongs to the commutant of $\mathcal G$ as $X$ commutes with $S_{\alpha}$
for all $\alpha\in\Lambda$.
This proves part (1) of the theorem.

To prove part (2), let $Y$ be in the commutant of the von Neumann algebra generated by the $R_\alpha$ and the $U_\alpha$.  Note that $R_\alpha$, $U_\alpha$ and $Y$ have the following
matrix representation with respect to the decomposition $\mathcal H \oplus (\mathcal K \ominus \mathcal H)$
$$
\left(
                 \begin{array}{cc}
                   S_\alpha & * \\
                   0 & * \\
                 \end{array}
               \right)\left(
                 \begin{array}{cc}
                   P_\alpha & * \\
                   0 & * \\
                 \end{array}
               \right)\text{ and }\left(
                 \begin{array}{cc}
                   X & * \\
                   * & * \\
                 \end{array}
               \right),
$$ respectively and they satisfy $R_\alpha^*YU_\alpha=R_\alpha Y$ and $U_\alpha^*YU_\alpha=Y$.
Now it follows from a simple block matrix computation that $X$, the compression of $Y$ to $\mathcal H$, is in $\mathcal T(\mathcal F)$.

Conversely, if $X$ is in $\mathcal T(\mathcal F)$, then the natural candidate for $Y$ is $Y=\pi(X)$. This indeed serves the purpose proving part (2).

To prove part (3), we first note that the representation $\pi_0$ in the statement of the theorem is actually the restriction of $\pi$ to $\mathcal C^*(\mathcal F)$ as the representation $\pi$
also maps the generating set $\mathcal F$ of $\mathcal C^*(\mathcal F)$ to the generating set $\mathcal G$ of $\mathcal C^*(\mathcal G)$. Since $\pi_0(\mathcal F)=\mathcal G$, range of $\pi_0$ is $\mathcal C^*(\mathcal G)$.
Therefore to prove that the following sequence
$$
0\rightarrow\mathcal I(\mathcal F)\hookrightarrow \mathcal C^*(\mathcal F)\xrightarrow{\pi_0} \mathcal C^*(\mathcal G)\rightarrow 0
$$
is a short exact sequence, all we need to show is that ker$\pi_0=\mathcal I(\mathcal F)$.
Since $\pi_0(\mathcal C^*(\mathcal F))$ is commutative, we have $XY-YX$ in the kernel of $\pi_0$, for any $X,Y\in \mathcal C^*(\mathcal F)\cap \mathcal T(\mathcal F)$.
Hence $\mathcal I(\mathcal F)\subseteq$ ker$\pi_0$. To prove the other inclusion, let us agree to denote by $\mathcal F^*$,
for a family $\mathcal F$ of operators, the adjoints of members of $\mathcal F$.
Let $Z_1$ be a finite product of members of $\mathcal F^*$ and $Z_2$ be a finite product of members of $\mathcal F$ and call $Z=Z_1Z_2$.
Then by the commutativity of the family $\mathcal F$, we have for each $\alpha \in \Lambda$, $\Phi_\alpha(Z)=Z$ and hence $\Phi_0(Z)=Z$,
where $\Phi_0$ and $\Phi_\alpha$'s are as in the proof of part (1). Note that $\Phi_0(Z)=P_{\mathcal H}\pi_0(Z)|_{\mathcal H}$,
for every $Z\in \mathcal C^*(\mathcal F)$. Now let $Z$ be any arbitrary finite product of members from $\mathcal F$ and $\mathcal F^*$.
Since $\pi_0(\mathcal F)=\mathcal G$, which is a family of normal operators, we obtain, by Fuglede-Putnam's theorem that,
action of $\Phi_0$ on $Z$ has all the members from $\mathcal F^*$ at the left and all the members from $\mathcal F$ at the right.
It follows from $\text{ker}\pi=\text{ker}\Phi$ and $\Phi$ is idempotent that ker$\pi_0=\{X-\Phi_0(X):X\in \mathcal C^*(\mathcal F)\}$. Also, because of the above description of $\Phi_0(X)$, if
$X$ is a finite product of elements from $\mathcal F$ and $\mathcal F^*$ then $X-\Phi_0(X)$ belongs to the ideal generated by all the commutators $XY-YX$,
where $X,Y \in \mathcal C^*(\mathcal F)\cap \mathcal T(\mathcal F)$. This shows that ker$\pi_0=\mathcal I(\mathcal F)$.
In order to find a completely isometric cross section, recall the completely isometric map $\rho:\pi(\mathcal C^*(\mathcal T(\mathcal P)))\to \mathcal B(\mathcal H)$ such that $\pi\circ \rho=id_{\pi(\mathcal C^*(\mathcal T(\mathcal P)))}$. Set $\rho_0:=\rho|_{\pi(\mathcal C^*(\mathcal F))}$. Then by the definition of $\rho$ it follows that
$\text{Ran} \rho_0\subseteq \mathcal C^*(\mathcal F)$ and therefore is a
completely isometric cross section.
This completes the proof of the theorem.
\end{proof}

\section{Dual Toeplitz Operators}
To pick up from where the last section ended, we note that as a special case of part (2) above, we know that if $(S,P)$ is a $\Gamma$-isometry on $\mathcal H$ with $(R,U)$ on $\mathcal K$ being its minimal $\Gamma$-unitary extension then an $X$ in $\mathcal B(\mathcal H)$ satisfies the Brown-Halmos relations with respect to $(S,P)$ if and only if there exists an operator $Y$ in the commutant of the von-Neumann algebra generated by $\{R,U\}$ such that $X=P_{\mathcal H}Y|_\mathcal H.$
The block matrix representation of the operator $Y$ shows that it need neither be an extension nor a co-extension of the operator $X$, in general. For example, choose the $\Gamma$-isometry to be $(T_s,T_p)$. Then by Theorem \ref{toepcharcthm}, any operator that satisfies the Brown-Halmos relations with respect
to this $\Gamma$-isometry is a Toeplitz operator with some symbol $\varphi \in L^\infty(b\Gamma)$ and $Y$, by Lemma \ref{lauraoperator}, would be $M_\varphi$, which has the matrix representation as in (\ref{matrix-of-M_phi}).

Dual Toeplitz operators have been studied on the Bergman space of the unit disc $\mathbb D$ in \cite{SZ} and on the Hardy space of the Euclidean ball $\mathbb B_d$ in \cite{DE} and \cite{Guediri}. In our setting, consider the space
\begin{align}\label{H^2-Decomp}
H^2(\mathbb G)^\perp = L^2(b\Gamma) \ominus H^2(\mathbb G).
\end{align}
For a symbol $\varphi \in L^\infty (b\Gamma)$,
define the {\em{dual Toeplitz operator}} on $H^2(\mathbb G)^\perp$ by $$ DT_\varphi = (I-Pr) M_\varphi|_{H^2(\mathbb G)^\perp},$$ where $(I-Pr)$ denotes the orthogonal projection of $L^2(b\Gamma)$ onto $H^2(\mathbb G)^\perp$. Therefore with respect to the decomposition (\ref{H^2-Decomp}) of $L^2(b\Gamma)$,
\begin{eqnarray}\label{matrix-of-M_phi} M_\varphi = \left(
                 \begin{array}{cc}
                   T_\varphi & H_{\overline{\varphi}}^* \\
                   H_\varphi & DT_\varphi \\
                 \end{array}
               \right).\end{eqnarray}

\begin{lemma}
The special pair $D = (DT_\sbar, DT_\pbar)$ is a $\Gamma$-isometry with $(M_\sbar,M_\pbar)$ as its minimal $\Gamma$-unitary extension. \end{lemma}
\begin{proof}
It is a $\Gamma$-isometry because it is the restriction of the $\Gamma$-unitary $(M_{\bar{s}},M_{\bar p})$ to the space $H^2(\mathbb G)^\perp$.
And this extension is minimal because $M_\pbar$ is the minimal unitary extension of $DT_\pbar$.
\end{proof}
\begin{theorem} A bounded operator $T$ on $H^2(\mathbb G)^\perp$ is a dual Toeplitz operator if and only if it satisfies the Brown-Halmos relations with respect to $D= (DT_\sbar, DT_\pbar)$. \end{theorem}
\begin{proof}
The fact that every dual Toeplitz operator on $H^2(\mathbb G)^\perp$ satisfies the Brown-Halmos relations with respect to $(DT_\sbar, DT_\pbar)$ follows from the following identities
$${M_\sbar}^*M_\varphi M_\pbar=M_\varphi M_\sbar \text{ and }{M_\pbar}^*M_\varphi M_\pbar=M_\varphi \text{ for every }\varphi \in L^\infty (b\Gamma)$$ and from the $2\times 2$ matrix representations of the operators in concern. For the converse, let $T$ on $H^2(\mathbb G)^\perp$ satisfy the Brown-Halmos relations with respect to the $\Gamma$-isometry $(DT_\sbar, DT_\pbar)$. By the comments at the beginning of this section and Lemma~\ref{lauraoperator},
there is a $\varphi \in L^\infty (b\Gamma)$ such that $T$ is the compression of $M_{\varphi}$ to $H^2(\mathbb G)^\perp$.
\end{proof}

\vspace{0.1in} \noindent\textbf{Acknowledgement:}

The first named author's research is supported by the University Grants Commission Centre for Advanced Studies. The research works of the second and third named authors are supported by DST-INSPIRE Faculty Fellowships DST/INSPIRE/04/2015/001094 and DST/INSPIRE/04/2018/002458 respectively. The third named author thanks Indian Institute of Technology, Bombay for a post-doctoral fellowship under which most of this work was done.

We are very grateful to the referees for substantial suggestions which improved the paper.

\end{document}